\documentclass[11pt, a4paper]{amsart}

\usepackage{amsmath,amssymb,amsthm}
\usepackage{hyperref}
\usepackage[dvipsnames,usenames]{color}
\usepackage[latin1]{inputenc}
\usepackage{graphicx}
\usepackage{psfrag}
\usepackage{color}
\usepackage{soul}

\newtheorem{guia}{}
\newtheorem{rem}{}

\newtheorem{teorema}[guia]{Theorem}

\newtheorem{lema}[guia]{Lemma}

\newtheorem{obss}[rem]{\it Remarks}

\newcommand{\al}{\alpha}
\newcommand{\be}{\beta}

\newcommand{\De}{\Delta}
\newcommand{\de}{\delta}
\newcommand{\ds}{\displaystyle}
\newcommand{\e}{\varepsilon}

\newcommand{\la}{\lambda}
\newcommand{\La}{\Lambda}
\newcommand{\N}{\mathbb N}

\newcommand{\Om}{\Omega}
\newcommand{\Omb}{\overline{\Omega}}
\newcommand{\p}{\partial}

\newcommand{\R}{\mathbb R}

\begin{document}

\title[A priori bounds and existence of solutions]{A priori bounds and existence of solutions
for some nonlocal elliptic problems}

\author[B. Barrios, L. Del Pezzo, J. Garc\'{\i}a-Meli\'{a}n and A. Quaas]
{B. Barrios, L. del Pezzo, J. Garc\'{\i}a-Meli\'{a}n\\ and A. Quaas}

\date{}

\address{B. Barrios \hfill\break\indent
Department of Mathematics, University of Texas at Austin
\hfill\break\indent Mathematics Dept. RLM 8.100 2515 Speedway Stop C1200
\hfill\break\indent  Austin, TX 78712-1202, USA.}
\email{{\tt bego.barrios@uam.es}}

\address{L. Del Pezzo \hfill\break\indent
CONICET  \hfill\break\indent
Departamento de Matem\'{a}tica, FCEyN UBA
\hfill\break\indent Ciudad Universitaria, Pab I (1428)
\hfill\break\indent Buenos Aires,
ARGENTINA. }
\email{{\tt ldpezzo@dm.uba.ar}}

\address{J. Garc\'{\i}a-Meli\'{a}n \hfill\break\indent
Departamento de An\'{a}lisis Matem\'{a}tico, Universidad de La Laguna
\hfill \break \indent C/. Astrof\'{\i}sico Francisco S\'{a}nchez s/n, 38271 -- La Laguna, SPAIN
\hfill\break\indent
{\rm and} \hfill\break
\indent Instituto Universitario de Estudios Avanzados (IUdEA) en F\'{\i}sica
At\'omica,\hfill\break\indent Molecular y Fot\'onica,
Universidad de La Laguna\hfill\break\indent C/. Astrof\'{\i}sico Francisco
S\'{a}nchez s/n, 38203 -- La Laguna, SPAIN.}
\email{{\tt jjgarmel@ull.es}}

\address{A. Quaas\hfill\break\indent
Departamento de Matem\'{a}tica, Universidad T\'ecnica Federico Santa Mar\'{\i}a
\hfill\break\indent  Casilla V-110, Avda. Espa\~na, 1680 --
Valpara\'{\i}so, CHILE.}
\email{{\tt alexander.quaas@usm.cl}}


\begin{abstract}
In this paper we show existence of solutions for some elliptic problems with nonlocal diffusion
by means of nonvariational tools. Our proof is based on the use of topological
degree, which requires a priori bounds for the solutions. We obtain the a priori
bounds by adapting the classical scaling method of Gidas and Spruck. We also
deal with problems involving gradient terms.
\end{abstract}

\maketitle

\section{Introduction}
\setcounter{section}{1}
\setcounter{equation}{0}

Nonlocal diffusion problems have received considerable attention during the last
years, mainly because their appearance when modelling different situations.
To name a few, let us mention anomalous diffusion and quasi-geostrophic flows, turbulence and
water waves, molecular dynamics and relativistic quantum mechanics of stars
(see \cite{BoG,CaV,Co,TZ} and references therein). They also appear in mathematical
finance (cf. \cite{A,Be,CoT}), elasticity problems \cite{signorini}, 
thin obstacle problem \cite{Caf79}, phase transition \cite{AB98, CSM05, SV08b},  crystal dislocation \cite{dfv, toland} 
and stratified materials \cite{savin_vald}.

A particular class of nonlocal operators which have been widely analyzed is given, up to 
a normalization constant, by
$$
(-\De)^s_K u(x) = \int_{\R^N} \frac{2u(x) -u(x+y)-u(x-y)}{|y|^{N+2s}} K(y) dy,
$$
where $s\in (0,1)$ and $K$ is a measurable function defined in $\R^N$ ($N\ge 2$). A remarkable
example of such operators is obtained by setting $K=1$, when $(-\De)^s_K$ reduces to the
well-known fractional Laplacian (see \cite[Chapter 5]{Stein} or \cite{NPV, Landkof, S} 
for further details). Of course, we will
require the operators $(-\De)^s_K$ to be elliptic, which in our context means that there exist positive
constants $\la\le \Lambda$ such that
\begin{equation}\label{elipticidad}
\la \le K(x) \le \Lambda \quad \hbox{in } \R^N
\end{equation}
(cf. \cite{CS}). While there is a large literature dealing with this class of operators,
very little is known about existence of solutions for nonlinear problems, except
for cases where variational methods can be employed (see for instance \cite{barrios2, barrios4, barrios3, SV2, servadeivaldinociBN, servadeivaldinociBNLOW} 
and references therein).

But when the problem under consideration is not of variational type, for instance when
gradient terms are present, as far as we know, results about existence of solutions are very scarce in the
literature.
Thus our objective is to find a way to show existence of solutions for some problems
under this assumption. For this aim, we will resort to the use of the fruitful topological
methods, in particular Leray-Schauder degree.

It is well-known that the use of these methods requires the knowledge of the so-called
a priori bounds for all possible solutions. Therefore we will be
mainly concerned with the obtention of these a priori bounds for a particular class of
equations. A natural starting point for this program is to consider the problem:
\begin{equation}\label{problema}
\left\{
\begin{array}{ll}
(-\De)_K^s u = u^p + g(x,u) & \hbox{in }\Om,\\[0.35pc]
\ \ u=0 & \hbox{in }\R^N \setminus \Om,
\end{array}
\right.
\end{equation}
where $\Om \subset \R^N$ is a smooth bounded domain, $p>1$ and $g$ is a perturbation term which is
small in some sense. Under several expected restrictions on $g$ and $p$ we will show
that all positive solutions of this problem are a priori bounded. The most important
requirement is that $p$ is subcritical, that is
\begin{equation}\label{subcritico}
1<p< \frac{N+2s}{N-2s}
\end{equation}
and that the term $g(x,u)$ is a small perturbation of $u^p$ at infinity. By adapting the
classical scaling method of Gidas and Spruck (\cite{GS}) we can show that all positive solutions 
of \eqref{problema} are a priori bounded.

An important additional assumption that we will be imposing on the kernel $K$ is that
\begin{equation}\label{continuidad}
\lim_{x\to 0} K(x)=1.
\end{equation}

\medskip

It is important to clarify at this moment that we are always dealing with
viscosity solutions $u \in C(\R^N)$
in the sense of \cite{CS}, although in some cases the solutions will turn out to be more regular with the
help of the regularity theory developed in \cite{CS, CS2}.

With regard to problem \eqref{problema}, our main result is the following:

\begin{teorema}\label{th-1}
Assume $\Om$ is a $C^2$ bounded domain of $\R^N$, $N\ge 2$, $s\in (0,1)$ and $p$ verifies
\eqref{subcritico}. Let $K$ be a measurable kernel that satisfies \eqref{elipticidad} and \eqref{continuidad}.
If $g\in C(\Omb \times \R)$ verifies
$$
|g(x,z)| \le C |z|^r \qquad x\in \Omb, \ z\in \R,
$$
where $1<r<p$, then problem \eqref{problema} admits at least a positive
viscosity solution.
\end{teorema}

It is to be noted that the scaling method requires on one side of good estimates for solutions,
both interior and at the boundary, and on the other side of a Liouville theorem in $\R^N$.
In the present case interior estimates are well known (cf. \cite{CS}), but good local estimates
near the boundary do not seem to be available. We overcome this problem by constructing
suitable barriers which can be controlled when the scaled domains are moving. It is worthy of
mention at this point that the corresponding Liouville theorems are already available (cf.
\cite{ZCCY,CLO1,QX,FW}).

Let us also mention that we were not aware of any work dealing with the question of a priori
bounds for problem \eqref{problema}; however, when we were completing this manuscript, it has
just come to our attention the very recent preprint \cite{CLC}, where a priori bounds for
smooth solutions are obtained in problem \eqref{problema} with $K=1$ and $g=0$ (but no
existence is shown).
On the other hand, it is important to mention the papers \cite{BCPS,CT,CZ,C},
where a priori bounds and Liouville results have been obtained for related operators, like
the ``spectral" fractional laplacian. To see some diferences between this operator and $(-\Delta)^s$, 
obtained by setting $K=1$ in the present work, see for instance \cite{SV}.
In all the previous works dealing with the spectral fractional Laplacian, the main tool is the well-known
Caffarelli-Silvestre extension obtained in \cite{CS3}. This tool is not available for us here, 
hence we will treat the problem in a nonlocal way with a direct approach.

\medskip

As we commented before, we will also be concerned with the adaptation of the previous result to some more general
equations. More precisely, we will study the perturbation of equation
\eqref{problema} with the introduction of gradient terms, that is,
\begin{equation}\label{problema-grad}
\left\{
\begin{array}{ll}
(-\De)_K^s u = u^p + h(x,u,\nabla u) & \hbox{in }\Om,\\[0.35pc]
\ \ u=0 & \hbox{in }\R^N \setminus \Om.
\end{array}
\right.
\end{equation}
For the type of nonlocal equations that we are analyzing, a natural restriction in order
that the gradient is meaningful is $s>\frac{1}{2}$. However, there seem to be few works
dealing with nonlocal equations with gradient terms (see for example 
\cite{AI,BCI,BK2,CaV,CL,CV,GJL,S2,SVZ,W}). 

It is to be noted that, at least in the case $K=1$, since solutions $u$
are expected to behave
like ${\rm dist}(x,\p\Om)^s$ near the boundary by Hopf's principle
(cf. \cite{ROS}), then the
gradient is expected to be singular near $\p\Om$. This implies that the standard scaling
method has to be modified to take care of this singularity. We achieve this by introducing
some suitable weighted norms which have been already used in the context of second order
elliptic equations (cf. \cite{GT}).

However, the introduction of this weighted norms presents some problems since the
scaling needed near the boundary is not the same one as in the interior. Therefore we
need to split our study into two parts: first, we obtain ``rough" universal bounds for all
solutions of \eqref{problema-grad}, by using the well-known doubling lemma in \cite{PQS}.
Since our problems are nonlocal in nature this forces us to strengthen the 
subcriticality hypothesis \eqref{subcritico} and to require instead
\begin{equation}\label{subserrin}
1<p< \frac{N}{N-2s}
\end{equation}
(cf. Remarks \ref{comentario} (b) in Section 3).
After that, we reduce the obtention of the a priori bounds to an analysis near the boundary.
With a suitable scaling, the lack of a priori bounds leads to a problem in a half-space which
has no solutions according to the results in \cite{QX} or \cite{FW}.

It is worth stressing that the main results in this paper rely in the construction 
of suitable barriers for equations with a singular right-hand side, which are well-behaved with respect to
suitable perturbations of the domain (cf. Section \ref{s2}).

\medskip

Le us finally state our result for problem \eqref{problema-grad}. In this context, a solution of
\eqref{problema-grad} is a function $u\in C^1(\Om)\cap C(\R^N)$ vanishing outside $\Om$ and 
verifying the equation in the viscosity sense.

\begin{teorema}\label{th-grad}
Assume $\Om$ is a $C^2$ bounded domain of $\R^N$, $N\ge 2$, $s \in (\frac{1}{2},1)$ and $p$ verifies
\eqref{subserrin}. Let $K$ be a measurable kernel that satisfies \eqref{elipticidad} and \eqref{continuidad}. If
$h\in C(\Om \times \R\times \R^N)$ is nonnegative and verifies
$$
h(x,z,\xi) \le C (|z|^r + |\xi|^t), \quad x\in\Om,\ z\in \R,\ \xi\in \R^N,
$$
where $1<r<p$ and $1<t<\frac{2sp}{p+2s-1}$, then problem \eqref{problema-grad} admits at
least a positive solution.
\end{teorema}

\medskip

The rest of the paper is organized as follows: in Section 2 we recall some interior regularity
results needed for our arguments, and we solve some linear problems by constructing
suitable barriers. Section 3 is dedicated to the obtention of a priori bounds, while in Secion 4
we show the existence of solutions that is, we give the proofs of Theorems \ref{th-1} and  
\ref{th-grad}.

\medskip

\section{Interior regularity and some barriers}\label{s2}
\setcounter{section}{2}
\setcounter{equation}{0}

The aim of this section is to collect several results regarding the construction of
suitable barriers and also some interior regularity for equations related to \eqref{problema}
and \eqref{problema-grad}. We will use throughout the standard convention that the letter $C$ denotes 
a positive constant, probably different from line to line.

\medskip

Consider $s\in (0,1)$, a measurable kernel $K$ verifying \eqref{elipticidad}
and \eqref{continuidad} and a $C^2$ bounded domain $\Om$. We begin by analyzing the linear equation
\begin{equation}\label{eq-regularidad}
(-\De)^s_K u = f \quad \hbox{in } \Om,
\end{equation}
where $f\in L^\infty_{\rm loc}(\Om)$. As a consequence of Theorem 12.1 in
\cite{CS} we get that if $u \in C(\Om)\cap L^\infty( \R^N)$ is a viscosity solution of \eqref{eq-regularidad} then
$u\in C^\al_{\rm loc}(\Om)$ for some $\al \in (0,1)$. Moreover, for every ball $B_R\subset \subset \Om$
there exists a positive constant $C=C(N,s,\la,\La,R)$ such that:
\begin{equation}\label{est-ca}
\| u\|_{C^{\al}(\overline{B_{R/2}})} \le C\| f\| _{L^\infty(B_R)} + \| u\|_{L^\infty(\R^N)}.
\end{equation}
The precise dependence of the constant $C$ on $R$ can be determined by means of a simple scaling,
as in Lemma \ref{lema-regularidad} below; however, for interior estimates this will be of no
importance to us.
When $s>\frac{1}{2}$, the H\"older estimate for the solution can be improved to obtain an
estimate for the first derivatives. In fact, as a consequence of Theorem 1.2 in \cite{K}, we have that
$u\in C^{1,\be}_{\rm loc}(\Om)$, for some $\beta=\beta(N,s,\la,\La)  \in (0,1)$.  Also, for every
ball $B_R\subset \subset \Om$ there exists a positive constant $C=C(N,s,\la,\La,R)$ such that:
\begin{equation}\label{est-c1a}
\| u\|_{C^{1,\beta}(\overline{B_{R/2}})} \le C
\left( \| f\| _{L^\infty(B_R)} + \| u\|_{L^\infty(\R^N)}\right).
\end{equation}
Both estimates will play a prominent role in our proof of a priori bounds for positive solutions of
\eqref{problema} and \eqref{problema-grad}.

\medskip

Next we need to deal with problems with a right hand side which is possibly singular at
$\p \Om$. For this aim, it is convenient to introduce some norms which will help us to quantify
the singularity of both the right hand sides and the gradient of the solutions in case
$s>\frac{1}{2}$.

Let us denote, for $x\in \Om$, $d(x)={\rm dist}(x,\p \Om)$. It is well known that $d$ is Lipschitz
continuous in $\Om$ with Lipschitz constant 1 and it is a $C^2$ function in a neighborhood of $\p\Om$.
We modify it outside this neighborhood to make it a $C^2$ function (still with Lipschitz constant
1), and we extend it to be zero outside $\Om$.

Now, for $\theta \in \R$ and $u \in C(\Om)$, let us denote (cf. Chapter 6 in \cite{GT}):
$$
\| u\|_0^{(\theta)} =\sup_\Om\; d(x)^{\theta} |u(x)|.
$$
When $u\in C^1(\Om)$ we also set
\begin{equation}\label{norma-c1}
\| u\|_1^{(\theta)} = \sup_\Om\; \left(d(x)^{\theta} |u(x)|+ d(x)^{\theta+1} |\nabla u(x)|\right).
\end{equation}
Then we have the following existence result for the Dirichlet problem associated to \eqref{eq-regularidad}.

\begin{lema}\label{lema-existencia}
Assume $\Om$ is a $C^2$ bounded domain, $0<s<1$ and $K$ is a measurable function
verifying \eqref{elipticidad} and \eqref{continuidad}. Let $f\in C(\Om)$ be such that
$\| f \|_0^{(\theta)}<+\infty$ for some $\theta \in (s,2s)$. Then the problem
\begin{equation}\label{prob-lineal}
\left\{
\begin{array}{ll}
(-\De)_K^s u=f  & \hbox{in } \Om,\\[0.35pc]
\; \; u=0 & \hbox{in } \R^N\setminus \Om,
\end{array}
\right.
\end{equation}
admits a unique viscosity solution. Moreover, there exists a positive constant $C$
such that
\begin{equation}\label{est-1}
\| u \|_0^{(\theta-2s)} \le C \| f \| _0^{(\theta)}.
\end{equation}
Finally, if $f\ge 0$ in $\Om$ then $u\ge 0$ in $\Om$.
\end{lema}

\bigskip

The proof of this result relies in the construction of a suitable barrier in a neighborhood of
the boundary of $\Om$ which we will undertake in the following lemma. This barrier
will also turn out to be important to obtain
bounds for the solutions when trying to apply the scaling method. It is worthy of mention
that for quite general operators, the lemma below can be obtained provided that
$\theta$ is taken close enough to $2s$ (cf. for instance Lemma 3.2 in \cite{FQ}). But
the precise assumptions we are imposing on $K$, especifically \eqref{continuidad},
allow us to construct the barrier in
the whole range $\theta \in (s,2s)$.

In what follows, we denote, for small positive $\de$,
$$
\Om_\de=\{x\in \Om: \hbox{dist}(x,\p \Om)<\de\}, 
$$
and $K_\mu(x)= K(\mu x)$ for $\mu>0$.

\begin{lema}\label{lema-barrera-1}
Let $\Om$ be a $C^2$ bounded domain of $\R^N$, $0<s<1$ and $K$ be measurable and
verify \eqref{elipticidad} and \eqref{continuidad}.
For every $\theta \in (s,2s)$ and $\mu_0>0$, there exist
$C_0,\de>0$ such that
$$
(-\De)^s_{K_\mu} d^{2s-\theta} \ge C_0 d^{-\theta} \quad \hbox{in } \Om_\de,
$$
if $0<\mu \le \mu_0$.
\end{lema}

\begin{proof}
By contradiction, let us assume that the conclusion of the lemma is not true. Then there exist $\theta\in
(s,2s)$, $\mu_0>0$,  sequences of points $x_n\in \Om$ with
$d(x_n)\to 0$ and numbers
$\mu_n\in (0,\mu_0]$ such that
\begin{equation}\label{contradiction}
\lim_{n\to +\infty} d(x_n)^{\theta} (-\De)^s_{K_{\mu_n}} d^{2s-\theta} (x_n) \le 0.
\end{equation}
Denoting for simplicity $d_n:=d(x_n)$, and performing the change of variables $y= d_n z$
in the integral appearing in \eqref{contradiction} we obtain
\begin{equation}\label{contradiction-2}
\int_{\R^N} \frac{2 - \left(\frac{d(x_n+d_n z)}{d_n}\right)^{2s-\theta}-
\left(\frac{d(x_n-d_n z)}{d_n}\right)^{2s-\theta}}{|z|^{N+2s}} K(\mu_n d_n z) dz \le o(1).
\end{equation}
Before passing to the limit in this integral, let us estimate it from below. Observe that when
$x_n+d_n  z\in \Om$, we have by the Lipschitz property of $d$ that $d(x_n+d_n z) \le d_n (1+|z|)$.
Of course, the same is true when $x_n +d_n z \not\in \Om$ and it similarly follows that $d(x_n-d_n z)
\le d_n (1+|z|)$. Thus, taking $L>0$ we obtain for large $n$
\begin{equation}\label{ineq1}
\begin{array}{l}
\ds \int_{|z| \ge L}
\frac{2 - \left(\frac{d(x_n+d_n z)}{d_n}\right)^{2s-\theta}- \left(\frac{d(x_n-d_n z)}{d_n}\right)^{2s-\theta}}
{|z|^{N+2s}} K(\mu_n d_n z) dz\\[1.4pc]
\quad \ge \ds - 2 \Lambda \int_{ |z| \ge L} \frac{(1+|z|)^{2s-\theta}}{|z|^{N+2s}} dz.
\end{array}
\end{equation}
On the other hand, since $d$ is smooth in a neighborhood of the boundary, when $|z|\le L$
and $x_n +d_n z\in \Om$, we obtain by Taylor's theorem
\begin{equation}\label{taylor}
d(x_n + d_n z )= d_n + d_n \nabla d(x_n) z + \Theta_n(d_n,z) d_n^2 |z|^2,
\end{equation}
where $\Theta_n$ is uniformly bounded. Hence
\begin{equation}\label{eq3}
d(x_n+d_n z ) \le d_n + d_n \nabla d(x_n) z + C d_n^2 |z|^2.
\end{equation}
Now choose $\eta \in (0,1)$ small enough. Since $d(x_n)\to 0$ and $|\nabla d|=1$ in a
neighborhood of the boundary, we can assume that
\begin{equation}\label{extra1}
\nabla d(x_n)\to e \hbox{ as }n\to +\infty \hbox{ for some unit vector }e.
\end{equation}
Without loss of generality, we may take
$e=e_N$, the last vector of the canonical basis of $\R^N$. If we restrict $z$ further to satisfy $|z|\le \eta$,
we obtain $1+\nabla d(x_n) z \sim 1 + z_N \ge 1-\eta>0$ for large $n$, since $|z_N| \le |z|\le \eta$.
Therefore, the right-hand side in \eqref{eq3} is positive for large $n$ (depending only on
$\eta$), so that the inequality \eqref{eq3} is also true when
$x_n+d_n z\not\in \Om$. Moreover, by using again Taylor's theorem
$$
(1+\nabla d(x_n) z +  C d_n |z|^2)^{2s-\theta} \le 1+ (2s-\theta) \nabla d(x_n) z  + C  |z|^2,
$$
for large enough $n$. Thus from \eqref{eq3},
$$
\left(\frac{d(x_n+d_n z)}{d_n}\right)^{2s-\theta} \le 1+ (2s-\theta) \nabla d(x_n) z  + C |z|^2,
$$
for large enough $n$. A similar inequality is obtained for the term involving $d(x_n-d_n z)$. Therefore we deduce
that
\begin{equation}\label{ineq2}
\begin{array}{l}
\ds \int_{ |z| \le \eta }
\frac{2 - \left(\frac{d(x_n+d_n z)}{d_n}\right)^{2s-\theta}- \left(\frac{d(x_n-d_n z)}{d_n}\right)^{2s-\theta}}{|z|^{N+2s}}
K(\mu_n d_n z) dz\\[1.4pc]
\quad \ge \ds - 2 \La C  \int_{ |z| \le \eta} \frac{1}{|z|^{N-2(1-s)}} dz.
\end{array}
\end{equation}
We finally observe that it follows from the above discussion (more precisely from \eqref{taylor} and
\eqref{extra1} with $e=e_N$) that for $\eta \le |z| \le L$
\begin{equation}\label{extra2}
\frac{d(x_n \pm d_n z)}{d_n} \to (1\pm z_N)_+ \qquad \hbox{as } n \to +\infty.
\end{equation}
Therefore using \eqref{ineq1}, \eqref{ineq2} and \eqref{extra2}, and passing to the limit as
$n\to +\infty$ in \eqref{contradiction-2}, by dominated convergence we arrive at
$$
\begin{array}{ll}
\ds - 2 \Lambda \int_{ |z| \ge L} \frac{(1+|z|)^{2s-\theta}}{|z|^{N+2s}} dz +
\int_{ \eta \le |z| \le L}  \frac{2 - (1+z_N)_+^{2s-\theta}- (1-z_N)_+^{2s-\theta}}{|z|^{N+2s}} dz \\[1.4pc]
\ds \qquad \qquad - 2 \La C  \int_{ |z| \le \eta} \frac{1}{|z|^{N-2(1-s)}} dz \le 0.
\end{array}
$$
We have also used that $\lim_{n\to +\infty} K(\mu_n d_n  z)=1$ uniformly, by \eqref{continuidad}
and the boundedness of $\{\mu_n\}$. Letting now $\eta \to 0$ and then $L\to +\infty$, we have
$$
\int_{ \R^N}  \frac{2 - (1+z_N)_+^{2s-\theta}- (1-z_N)_+^{2s-\theta}}{|z|^{N+2s}} dz \le 0.
$$
It is well-known, with the use of Fubini's theorem and a change of variables, that this
integral can be rewritten as a one-dimensional integral
\begin{equation}\label{contra-final}
\int_ \R \frac{2 - (1+t)_+^{2s-\theta}- (1-t)_+^{2s-\theta}}{|t|^{1+2s}} dt \le 0.
\end{equation}
We will see that this is impossible because of our assumption $\theta \in (s,2s)$. Indeed,
consider the function
$$
F(\tau) = \int_ \R \frac{ 2- (1+t)_+^\tau- (1-t)_+^\tau}{|t|^{1+2s}} dt, \quad \tau \in (0,2s),
$$
which is well-defined. We claim that $F \in C^\infty(0,2s)$ and it is strictly concave. In fact, observe that
for $k\in \N$, the candidate for the $k-$th derivative $F^{(k)}(\tau)$ is given by
$$
- \int_ \R \frac{(1+t)_+^\tau (\log(1+t))_+^k + (1-t)_+^\tau (\log(1-t))_+^k}{|t|^{1+2s}} dt.
$$
It is easily seen that this integral converges for every $k\ge 1$, since by Taylor's expansion for $t\sim 0$ we
deduce $(1+t)^\tau (\log(1+t))^k + (1-t)^\tau (\log(1-t))^k =O( t^2)$. Therefore it follows that
$F$ is $C^\infty$ in $(s,2s)$. To see that $F$ is strictly concave, just notice that
$$
F_\e''(\tau)= - \int_\R \frac{ (1+t)_+^\tau (\log (1+t)_+ )^2+ (1-t)_+^\tau (\log (1-t)_+ )^2}{|t|^{1+2s}}
dt < 0.
$$
Finally, it is clear that $F(0)=0$. Moreover, since $v(x)=(x_+)^s$, $x\in \R$ verifies $(-\De)^s v=0$ in $\R_+$
(see for instance the introduction in \cite{CJS} or Proposition 3.1 in \cite{ROS}), we also deduce that $F(s)=0$.
By strict concavity we have $F(\tau)>0$ for $\tau\in (0,s)$, which clearly contradicts \eqref{contra-final} if
$\theta\in (s,2s)$. Therefore \eqref{contra-final} is not true and this concludes the proof of the lemma.
\end{proof}

\bigskip

\begin{proof}[Proof of Lemma \ref{lema-existencia}]
By Lemma \ref{lema-barrera-1} with $\mu_0=1$, there exist $C_0>0$ and $\de>0$
such that
\begin{equation}\label{extra3}
\mbox{$(-\De)^s_K d^{2s-\theta} \ge C_0 d^{-\theta}$ in $\Om_\de$.}
\end{equation}
Let us show
that it is possible to construct a supersolution of the problem
\begin{equation}\label{supersolucion}
\left\{
\begin{array}{ll}
(-\De)_K^s v=C_0 d^{-\theta}  & \hbox{in } \Om,\\[0.35pc]
\; \; v=0 & \hbox{in } \R^N\setminus \Om,
\end{array}
\right.
\end{equation}
vanishing outside $\Om$.

First of all, by Theorem 3.1 in \cite{FQ}, there exists a nonnegative
function $w\in C(\R^N)$ such that $(-\De)_K^s w = 1$ in $\Om$,
with $w=0$ in $\R^N \setminus \Om$.
We claim that $v= d^{2s-\theta} + t w$ is a supersolution of \eqref{supersolucion}
if $t>0$ is large enough. For this aim, observe that $(-\De)^s_K d^{2s-\theta} \ge -C$ in $\Om \setminus
\Om_\de$, since $d$ is a $C^2$ function there. Therefore,
$$
\mbox{$(-\De)^s_K v\ge t-C \ge C_0 d^{-\theta}$
in $\Om\setminus \Om_\de$}
$$
if $t$ is large enough. Since clearly $(-\De)^s_K v \ge C_0 d^{-\theta}$
in $\Om_\de$ as well, we see that $v$ is a supersolution of \eqref{supersolucion}, which vanishes
outside $\Om$.

Now choose a sequence of smooth functions $\{\psi_n\}$ verifying $0\le \psi_n\le 1$, $\psi_n=1$ in
$\Om \setminus \Om_{2/n}$ and $\psi_n=0$ in $\Om_{1/n}$. Define $f_n= f\psi_n$, and consider
the problem
\begin{equation}\label{perturbado}
\left\{
\begin{array}{ll}
(-\De)_K^s u= f_n  & \hbox{in } \Om,\\[0.35pc]
\; \; u=0 & \hbox{in } \R^N\setminus \Om.
\end{array}
\right.
\end{equation}
Since $f_n \in C(\Omb)$, we can use Theorem 3.1 in \cite{FQ} which gives a viscosity solution
$u_n\in C(\R^N)$ of \eqref{perturbado}.

On the other hand, $|f_n| \le |f| \le \| f \|_0^{(\theta)} d^{-\theta}$ in $\Om$, so that the functions
$v_{\pm} = \pm C_0^{-1} \| f \|_0^{(\theta)} v$ are sub and supersolution of \eqref{perturbado}.
By comparison (cf. Theorem 5.2 in \cite{CS}), we obtain
$$
- C_0^{-1} \| f\| _0^{(\theta)} v \le u_n \le C_0^{-1} \| f\| _0^{(\theta)}v  \qquad \hbox{in } \Om.
$$
Now, this bound together with \eqref{est-ca}, Ascoli-Arzel\'a's theorem and a standard diagonal
argument allow us to obtain a subsequence, still denoted by $\{u_n\}$, and a function $u\in C(\Om)$
such that $u_n\to u$ uniformly on compact sets of $\Om$. In addition, $u$ verifies
\begin{equation}\label{eq-ult}
|u |\le C_0^{-1} \| f\| _0^{(\theta)}v \quad \hbox{in }\Om.
\end{equation}
By Corollary 4.7 in \cite{CS}, we can pass to the limit in \eqref{perturbado}
to obtain that $u \in C(\mathbb{R}^{N})$ is a viscosity solution of \eqref{prob-lineal}.
Moreover inequality \eqref{eq-ult} implies that $|u |\le
C \| f\| _0^{(\theta)} d^{2s-\theta}$ in $\Om\setminus\Omega_\de$ for some $C>0$, so that,
by \eqref{prob-lineal}, \eqref{extra3} and the comparison principle, we obtain that
$$
|u| \le C  \| f\|_0^{\theta} d^{2s-\theta} \quad \hbox{in }\Om
$$
which shows \eqref{est-1}.

The uniqueness and the nonnegativity of $u$ when $f \ge 0$ are a consequence
of the maximum principle (again Theorem 5.2 in \cite{CS}). This concludes the proof.
\end{proof}

\bigskip

Our next estimate concerns the gradient of the solutions of \eqref{prob-lineal} when
$s>\frac{1}{2}$. The proof is more or less standard starting from \eqref{est-c1a}
(cf. \cite{GT}) but we include it for completeness

\begin{lema}\label{lema-regularidad}
Assume $\Om$ is a smooth bounded domain and $s>\frac{1}{2}$. There exists a constant $C_0$
which depends on $N,s, \la$ and $\Lambda$  but not on $\Om$ such that, for every $\theta \in (s,2s)$
and $f\in C(\Om)$ with $\| f \|_0^{(\theta)}<+\infty$ the unique solution
$u$ of \eqref{prob-lineal} verifies
\begin{equation}\label{est-adimensional}
\| \nabla u \|_0^{(\theta - 2s +1)} \le C_0 ( \|f \|_0^{(\theta)} + \|u \|_{0}^{(\theta - 2s)}).
\end{equation}
\end{lema}

\begin{proof}
By \eqref{est-c1a} with $R=1$ we know that if
$(-\De)^s_K u=f$ in $B_1$ then there exists a constant which depends on $N,s, \la$ and
$\Lambda$ such that $\| \nabla u\|_{L^\infty (B_{1/2})} \le C ( \| f\| _{L^\infty(B_1)} + \| u\|_{L^\infty(\R^N)})$.
By a simple scaling, it can be seen that if $(-\De)^s_K u=f$ in $\Om$ and $B_R\subset \subset \Om$ then
$$
R \| \nabla u\|_{L^\infty (B_{R/2})} \le C ( R^{2s} \| f\| _{L^\infty(B_R)} + \| u\|_{L^\infty(\R^N)}).
$$
Choose a point $x\in \Om$. By applying the previous inequality in the ball $B=B_{d(x)/2}(x)$ and
multiplying by $d(x)^{\theta-2s}$ we arrive at
$$
d(x)^{\theta-2s+1} | \nabla u(x) | \le C \left( d(x)^\theta \| f\| _{L^\infty(B)} +
d(x)^{\theta-2s} \| u\|_{L^\infty(\R^N)}\right).
$$
Finally, notice that $\frac{d(x)}{2} < d(y) <\frac{3d(x)}{2}$ for every $y\in B$, so that $d(x)^\theta |f(y)|
\le 2^\theta d(y)^\theta f(y) \le 2^{2s} \| f \|_0^{(\theta)}$, this implying $d(x)^\theta \| f \|_{L^\infty(B)}
\le 2^{2s} \| f \|_0^{(\theta)}$. A similar inequality can be achieved for the term involving $\| u\|_{L^\infty(\R^N)}$.
After taking supremum, \eqref{est-adimensional} is obtained.
\end{proof}

\bigskip

Our next lemma is intended to take care of the constant in \eqref{est-1} when
we consider problem \eqref{prob-lineal} in expanding domains, since in general it depends on
$\Om$. This is the key for the
scaling method to work properly in our setting. For a $C^2$ bounded domain $\Om$, we take
$\xi \in \p\Om$, $\mu>0$ and let
$$
\mbox{$\Om^\mu:=\{y\in \R^N:\ \xi +\mu y\in \Om\}$.}
$$
It is clear then that
$d_\mu(y):={\rm dist}(y,\p \Om^\mu)=\mu^{-1} d(\xi +\mu y)$.
Let us explicitly
remark that the constant in \eqref{est-1} for the solution of
\eqref{prob-lineal} posed in $\Om^\mu$
will depend then on the domain $\Om$, but not on the dilation parameter $\mu$, as we show next.

\begin{lema}\label{lema-barrera-2}
Assume $\Om$ is a $C^2$ bounded domain, $0<s<1$ and $K$ is a measurable function verifying \eqref{elipticidad}
and \eqref{continuidad}. For every $\theta \in (s,2s)$ and $\mu_0>0$, there exist $C_0,\de>0$ such that
$$
(-\De)^s_{K_\mu} d_\mu ^{2s-\theta} \ge C_0 d_\mu^{-\theta} \quad \hbox{in } (\Om^\mu)_\de,
$$
if $0<\mu \le \mu_0$. Moreover, if $u$ verifies $(-\De)_{K_\mu}^s u \le C_1 d_\mu^{-\theta}$ in $\Om^\mu$
for some $C_1>0$ with $u=0$ in $\R^N\setminus \Om^\mu$, then
$$
u (x) \le C_2( C_1 +\|u  \|_{L^\infty(\Om^\mu)} )\; d_\mu ^{2s-\theta}
\quad \hbox{for } x \in (\Om^\mu)_\de.
$$
for some $C_2>0$ only depending on $s$, $\de$, $\theta$ and $C_0$.
\end{lema}

\begin{proof}
The first part of the proof is similar to that of Lemma \ref{lema-barrera-1} but taking a little more care
in the estimates. By contradiction let us assume that there exist sequences
$\xi_n\in \p \Om$, $\mu_n \in (0,\mu_0]$ and
$$
\mbox{$x_n \in \Om^n:=\{y\in \R^N:\ \xi_n + \mu_n y\in \Om\}$},
$$ 
such that $d_n(x_n)\to 0$ and
$$
d_n(x_n)^{\theta} (-\De)^s_{K_{\mu_n}} d_n^{2s-\theta} (x_n) \le o(1).
$$
Here we have denoted
$$
\mbox{$d_n(y):={\rm dist}(y,\p \Om^n) = \mu_n^{-1} d(\xi_n+\mu_n y)$.}
$$
For $L>0$, we obtain as in Lemma \ref{lema-barrera-1}, letting $d_n=d_n(x_n)$
$$
\begin{array}{l}
\ds \int_{|z| \ge L}
\frac{2 - \left(\frac{d_n(x_n+d_n z)}{d_n}\right)^{2s-\theta}- \left(\frac{d_n (x_n-d_n z)}{d_n}\right)^{2s-\theta}}
{|z|^{N+2s}} K(\mu_n d_n z) dz\\[1.4pc]
\quad \ge \ds - 2 \Lambda \int_{ |z| \ge L} \frac{(1+|z|)^{2s-\theta}}{|z|^{N+2s}} dz.
\end{array}
$$
Moreover, we also have an equation like \eqref{taylor}. In fact taking into account that
$\| D^2 d_n \| = \mu_n \|  D^2 d\|$ is bounded we have for $|z|\le \eta <1$:
$$
d_n (x_n \pm d_n z ) \le d_n \pm d_n \nabla d_n (x_n) z + C d_n^2 |z|^2.
$$
with a constant $C>0$ independent of $n$. Hence
$$
\begin{array}{l}
\ds \int_{ |z| \le \eta }
\frac{2 - \left(\frac{d_n(x_n+d_n z)}{d_n}\right)^{2s-\theta}- \left(\frac{d_n(x_n-d_n z)}{d_n}\right)^{2s-\theta}}{|z|^{N+2s}}
K(\mu_n d_n z) dz\\[1.4pc]
\quad \ge \ds - 2 \La C  \int_{ |z| \le \eta} \frac{1}{|z|^{N-2(1-s)}} dz.
\end{array}
$$
Now observe that $d_n(x_n)\to 0$ implies in particular $d(\xi_n+\mu_n x_n) \to 0$, so that
$|\nabla d(\xi_n+\mu_n x_n)|=1$ for large $n$ and then $|\nabla d_n (x_n)|=1$. As in \eqref{extra1},
passing to a subsequence we may assume that $\nabla d_n(x_n)\to e_N$. Then
$$
\frac{d_n(x_n \pm d_n z)}{d_n} \to (1\pm z_N)_+ \qquad \hbox{as } n \to +\infty,
$$
for $\eta \le |z| \le L$ and the proof of the first part concludes as in Lemma \ref{lema-barrera-1}.

\medskip

Now let $u$ be a viscosity solution of
$$\left\{
\begin{array}{ll}
(-\De)_{K_\mu}^s u \le C_1 d_\mu^{-\theta} & \hbox{in } \Om^\mu,\\[0.35pc]
\ \ u=0 & \hbox{in }\R^N\setminus \Om^\mu.
\end{array}
\right.
$$
Choose $R>0$ and let $v=R d_\mu^{2s-\theta}$. Then
clearly
$$
\mbox{$(-\De)^s_{K_\mu} v \ge RC_0 d_\mu ^{-\theta}\ge C_1 d_\mu^{-\theta} \ge (-\De)^s_{K_\mu} u$
in $(\Om^\mu)_\de$,}
$$
if we choose $R>C_1 C_0^{-1}$. Moreover, $u=v=0$ in $\R^N\setminus \Om^\mu$ and
$v \ge R \de^{2s-\theta} \ge u$ in $\Om^\mu \setminus (\Om^\mu)_\de$ if $R$ is chosen so
that $R \de^{2s-\theta} \ge \|u\|_{L^\infty(\Om^\mu)}$. Thus by comparison $u\le v$ in
$(\Om^\mu)_\de$, which gives the desired result, with, for instance $C_2=\de^{\theta-2s} +C_0^{-1}$.
This concludes the proof.
\end{proof}

\medskip

We close this section with a statement of the strong comparison principle for
the operator $(-\De)^s_K$, which will be frequently used throughout the
rest of the paper. We include a proof for completeness (cf. Lemma 12 in \cite{LL} for a similar proof).

\begin{lema}\label{PFM}
Let $K$ be a measurable function verifying \eqref{elipticidad} and assume $u\in C(\R^N)$,
$u\ge 0$ in $\R^N$ verifies $(-\De)^s_K u \ge 0$ in the viscosity sense in $\Om$. Then
$u>0$ or $u\equiv 0$ in $\Om$.
\end{lema}

\begin{proof}
Assume $u(x_0)=0$ for some $x_0\in \Om$ but $u\not\equiv 0$ in $\Om$. Choose a nonnegative test
function $\phi \in C^2(\R^N)$ such that $u\ge \phi$ in a neighborhood $U$ of $x_0$ with $\phi(x_0)=0$
and let
$$
\psi=\left\{
\begin{array}{ll}
\phi & \hbox{in } U\\
u & \hbox{in } \R^N\setminus U.
\end{array}
\right.
$$
Observe that $\psi$ can be taken to be nontrivial since $u$ is not identically zero, by diminishing $U$ if
necessary. Since $(-\De)^s_K u\ge 0$ in $\Om$ in the viscosity sense, it follows that $(-\De)^s_K
\psi (x_0)\ge 0$. Taking into account that for a nonconstant $\psi$
we should have $(-\De)^s_K \psi < 0$ at a global minimum, we deduce that
$\psi$ is a constant function.
Moreover, since $\psi(x_0)=\phi(x_0)=0$ then
 $\psi\equiv 0$ in $\R^N$, which is a contradiction.
Therefore if $u(x_0)=0$ for some $x_0\in \Om$ we must have $u\equiv 0$ in $\Om$, as was to be shown.
\end{proof}

\medskip

\section{A priori bounds}
\setcounter{section}{3}
\setcounter{equation}{0}

In this section we will be concerned with our most important step: the obtention of a priori
bounds for positive solutions for both problems \eqref{problema} and \eqref{problema-grad}.
We begin with problem \eqref{problema}, with the essential assumption of subcriticality of
$p$, that is equation \eqref{subcritico} and assuming that $g$ verifies the growth restriction
\begin{equation}\label{crec-g-2}
|g(x,z)| \le C (1 + |z|^r), \quad x\in\Om,\ z\in \R,
\end{equation}
where $C>0$ and $0<r<p$.
\medskip

\begin{teorema}\label{cotas}
Assume $\Om$ is a $C^2$ bounded domain and $K$ a measurable function verifying
\eqref{elipticidad} and \eqref{continuidad}. Suppose $p$ is such that \eqref{subcritico} holds
and $g$ verifies \eqref{crec-g-2}. Then there exists a constant $C>0$ such that for every positive
viscosity solution $u$ of \eqref{problema} we have
$$
 \| u\|_{L^\infty (\Om)} \le C.
$$
\end{teorema}

\begin{proof}
Assume on the contrary that there exists a sequence of positive solutions $\{u_k\}$ of \eqref{problema}
such that $M_k=\| u_k \|_{L^\infty(\Om)} \to +\infty$. Let $x_k\in \Om$ be points with
$u_k(x_k) =M_k $ and introduce the functions
$$
v_k(y)= \frac{u_k(x_k+\mu_k y)}{M_k}, \quad y\in \Om^k,
$$
where $\mu_k=M_k^{-\frac{p-1}{2s}}\to 0$ and
$$
\mbox{$\Om^k:=\{y\in \R^N:\ x_k+\mu_k y\in \Om\}$.}
$$
Then $v_k$ is a function verifying $0< v_k \le 1$, $v_k(0)=1$ and
\begin{equation}\label{rescale-1}
(-\De)^s_{K_k} v_k = v_k^p + h_k \quad \hbox{in } \Om^k
\end{equation}
where $K_k(y)=K(\mu_k y)$ and $h_k	\in C(\Om^k)$ verifies $|h_k|\le C M_k^{r-p}$.

By passing to subsequences, two situations may arise: either $d(x_k) \mu_k^{-1} \to +\infty$
or  $d(x_k) \mu_k^{-1} \to d \ge 0$.

\medskip

Assume the first case holds, so that $\Om^k \to \R^N$ as $k\to +\infty$. Since the right hand
side in \eqref{rescale-1} is uniformly bounded and $v_k\le 1$, we may use
estimates \eqref{est-ca} with an application of Ascoli-Arzel\'a's theorem and a diagonal argument
to obtain that $v_k \to v$ locally uniformly in $\R^N$. Passing to the limit in
\eqref{rescale-1} and using that $K$ is continuous at zero with $K(0)=1$, we see that
$v$ solves $(-\De)^s v= v^p$ in $\R^N$ in the viscosity sense (use for instance Lemma 5 in \cite{CS2}).

By standard regularity (cf. for instance Proposition 2.8 in \cite{S})
we obtain $v\in C^{2s+\al}(\R^N)$ for some $\al \in (0,1)$. Moreover, since $v(0)=1$, the strong
maximum principle implies $v>0$. Then by bootstrapping using again Proposition 2.8 in \cite{S}
we would actually have $v\in C^\infty(\R^N)$. In particular we deduce that $v$ is a strong
solution of $(-\De)^s v=v^p$ in $\R^N$ in the sense of \cite{ZCCY}. However,
since $p<\frac{N+2s}{N-2s}$, this contradicts for instance Theorem 4 in \cite{ZCCY} (see also \cite{CLO1}).

\medskip

If the second case holds then we may assume $x_k\to x_0\in \p\Om$. With no loss of generality
assume also $\nu (x_0)=-e_N$. In this case, rather than working
with the functions $v_k$, it is more convenient to deal with
$$
w_k(y)= \frac{u_k(\xi_k+\mu_k y)}{M_k}, \quad y\in D^k,
$$
where $\xi_k\in \p\Om$ is the projection of $x_k$ on $\p\Om$ and
\begin{equation}\label{Dk}
\mbox{$D^k:=\{y\in \R^N:\
\xi_k+\mu_k y \in \Om\}$.}
\end{equation}
Observe that
\begin{equation}\label{cero}
0\in \p D^k,
\end{equation}
and
$$\mbox{$D^k \to \R^N_+=\{y\in \R^N:\ y_N>0\}$ as $k\to +\infty$.}
$$
It also follows that $w_k$ verifies \eqref{rescale-1} in $D^{k}$
with a slightly different function $h_k$, but with the same bounds.

Moreover, setting
$$
y_k:=\frac{x_k-\xi_k}{\mu_k},
$$
so that $|y_k|= d(x_k)\mu_k^{-1}$, we see that
$w_k(y_k)=1$. We claim that $d=\lim_{k\to +\infty} d(x_k) \mu_k^{-1}>0$. This in particular guarantees that
by passing to a further subsequence $y_k\to y_0$, where $|y_0|=d>0$, thus $y_0$ is in the 
interior of the half-space $\R^N_+$.

\medskip

Let us show the claim. Observe that by \eqref{rescale-1}, and since $r<p$, we have 
$$
(-\Delta)^{s}_{K_k}w_k\leq C\leq C_1 d_k^{-\theta} \quad \hbox{in } D^k
$$ 
for every $\theta \in (s,2s)$, where $d_k(y)={\rm dist}(y,\p D^k)$. 
By Lemma \ref{lema-barrera-2}, fixing any such $\theta$, there exist
constants $C_0>0$ and $\de>0$ such that $w_k(y) \le C_0 d_k(y)^{2s-\theta}$ if $d_k(y) < \de$.
In particular, since by \eqref{cero} $|y_k|\ge d_k(y_k)$, 
if $d_k(y_k) <\de$, then $1\le C_0 d_k(y_k)^{2s-\theta} \le C_0
|y_k|^{2s-\theta}$, which implies $|y_k|$ is bounded from
below so that $d>0$.

Now we can employ 	\eqref{est-ca} as above to obtain that $w_k\to w$ uniformly on
compact sets of $\R^N_+$, where $w$ verifies $0\le w \le 1$ in $\R^N_+$, $w(y_0)=1$ and
$w(y) \le C y_N^{2s-\theta}$ for $y_N <\de$. Therefore $w\in C(\R^N)$ is a nonnegative, bounded
solution of
$$
\left\{
\begin{array}{ll}
(-\De)^s w = w^p & \hbox{in } \R^N_+,\\[0.25pc]
w=0 & \hbox{in } \R^N \setminus \R^N_+.
\end{array}
\right.
$$
Again by bootstrapping and the strong maximum principle we have $w\in C^\infty (\R^N_+)$, $w>0$.
Since $p<\frac{N+2s}{N-2s}<\frac{N-1+2s}{N-1-2s}$, this is a contradiction with Theorem 1.1 
in \cite{QX} (cf. also Theorem 1.2 in \cite{FW}). This contradiction proves the theorem.
\end{proof}

\bigskip

We now turn to analyze the a priori bounds for solutions of problem \eqref{problema-grad}.
We have already remarked that due to the expected singularity of the gradient of the solutions
near the boundary we need to work in spaces with weights which take care of the singularity.
Thus we fix $\sigma \in (0,1)$ verifying
\begin{equation}\label{cond-sigma}
0<\sigma< 1-\frac{s}{t}<1
\end{equation}
and let
\begin{equation}\label{E}
E_\sigma=\{u\in C^1(\Om): \ \| u\|_1^{(-\sigma)}<+\infty\},
\end{equation}
where $\| \cdot \|_1^{(-\sigma)}$
is given by \eqref{norma-c1} with $\theta=-\sigma$. As for the function $h$, we assume that it has a
prescribed growth at infinity: there exists $C^{0}>0$ such that for every $x\in\Om$, $z\in\R$ and $\xi\in\R^N$, 
\begin{equation}\label{crec-h-2}
|h(x,z,\xi)| \le C^0 (1 + |z|^r + |\xi|^t), 
\end{equation}
where $0<r<p$ and $1<t<\frac{2sp}{p+2s-1}<2s$ (observe that there is no loss of generality in 
assuming $t>1$). We recall that in the present situation we require the stronger 
restriction \eqref{subserrin} 
on the exponent $p$.

\medskip

Then we can prove:

\begin{teorema}\label{cotas-grad}
Assume $\Om$ is a $C^2$ bounded domain and $K$ a measurable function verifying
\eqref{elipticidad} and \eqref{continuidad}. Suppose that $s>\frac{1}{2}$, $p$ verifies \eqref{subserrin}
and $h$ is nonnegative and such that \eqref{crec-h-2}
holds. Then there exists a constant $C>0$ such that for every positive solution $u$ of \eqref{problema-grad}
in $E_\sigma$ with $\sigma$ satisfying \eqref{cond-sigma} we have
$$
 \| u\|_1^{(-\sigma)} \le C.
$$
\end{teorema}

\bigskip

We prove the a priori bounds in two steps. In the first one we obtain rough bounds for all solutions of
the equation which are universal, in the spirit of \cite{PQS}. It is here where the restriction
\eqref{subserrin} comes in. 

\begin{lema}\label{cotas-pqs}
Assume $\Om$ is a $C^2$ (not necessarily bounded) domain and $K$ a measurable function verifying
\eqref{elipticidad} and \eqref{continuidad}. Suppose that $s>\frac{1}{2}$ and $p$ verifies
\eqref{subserrin}. Then there exists a positive constant
$C=C(N,s,p,r,t,C^0,\Om)$
(where $r$, $t$ and $C^0$ are given in \eqref{crec-h-2})
such that for
every positive function $u\in C^1(\Om)\cap L^\infty(\R^N)$ verifying $(-\De)^s_K
u= u^p +h(x,u,\nabla u)$ in the viscosity sense in $\Om$, we have
$$
u(x) \le C (1+{\rm dist}(x,\p \Om)^{-\frac{2s}{p-1}}) ,\quad |\nabla u(x)| \le C (1+{\rm dist}(x,\p\Om)
^{-\frac{2s}{p-1}-1})
$$
for $x\in \Om$.
\end{lema}

\begin{proof}
Assume on the contrary that there exist sequences of
positive functions $u_k\in C^1(\Om)\cap L^\infty (\R^N)$ verifying $(-\De)^s_K u_k= u_k^p
+h(x,u_k,\nabla u_k)$ in $\Om$ and points $y_k\in \Om$ such that
\begin{equation}\label{hipo}
u_k(y_k)^\frac{p-1}{2s} + |\nabla u_k(y_k)|^\frac{p-1}{p+2s-1} > 2k\: (1+{\rm dist}(y_k,\p \Om)^{-1}).
\end{equation}
Denote $N_k(x)=u_k(x)^\frac{p-1}{2s} + |\nabla u_k(x)|^\frac{p-1}{p+2s-1}$, $x\in \Om$. By Lemma
5.1 in \cite{PQS} (cf. also Remark 5.2 (b) there) there exists a sequence of points $x_k\in \Om$ with the
property that $N_k(x_k) \ge N_k(y_k)$, $N_k(x_k)>2k\: {\rm dist}(x_k,\p \Om)^{-1}$ and
\begin{equation}\label{conjuntos}
\mbox{$N_k(z) \le 2 N_k(x_k)$ in $B(x_k, kN_k(x_k)^{-1})$.}
\end{equation}
Observe that, in particular, \eqref{hipo} implies that $N_k(x_k)\to +\infty$.
Let $\nu_k := N_k(x_k)^{-1}\to 0$ and
define
\begin{equation}\label{v_k}
v_k(y) :=  \nu_k^\frac{2s}{p-1} u_k (x_k+\nu _k y), \quad y \in B_k:=\{y\in \R^N:\ |y|<k\}.
\end{equation}
Then the functions $v_k$ verify $(-\De)^s_{K_k} v_k= v_k^p +  h_k$ in $B_k$, where
$K_k(y) = K (\mu_k y)$ and
$$
h_k (y)=\nu_k ^{\frac{2sp}{p-1}} h(\xi_k+\nu_k y ,\nu_k^{-\frac{2s}{p-1}}
v_k(y),\nu_k(x_k)^{-\frac{2s+p-1}{p-1}} \nabla v_k(y)).
$$
Since $h$ verifies \eqref{crec-h-2}, we have $| h_k| \le C_0 \nu_k^{\gamma}
(1+v_k^r+|\nabla v_k|^t)$ in $B_k$, where
$$
\gamma= \max\left\{ \frac{2s(p-r)}{p-1}, \frac{2ps-(2s+p-1)t}{p-1}\right\}  >0.
$$
Moreover by \eqref{conjuntos} it follows that
\begin{equation}\label{eq1}
v_k(y)^\frac{p-1}{2s} + |\nabla v_k(y)|^\frac{p-1}{p+2s-1}\le 2, \quad y\in B_k .
\end{equation}
Also it is clear that 
\begin{equation}\label{eq2}
v_k(0)^\frac{p-1}{2s} + |\nabla v_k(0)|^\frac{p-1}{p+2s-1} =1.
\end{equation}
Since $\nu_k\to 0$ and $v_k$ and $|\nabla v_k|$ are uniformly bounded in $B_k$, we see that
$ h_k$ is also uniformly bounded in $B_k$. We may then
use estimate \eqref{est-c1a} to obtain, again with the use of Ascoli-Arzel\'a's theorem
and a diagonal argument, that there exists a subsequence, still labeled $v_k$ such that
$v_k\to v$ in $C^1_{\rm loc}(\R^N)$ as $k\to +\infty$. Since $v(0)^\frac{p-1}{2s} + |\nabla
v(0)|^\frac{p-1}{p+2s-1} =1$, we see that $v$ is nontrivial. 

Now let $w_k$ be the functions obtained by extending $v_k$ to be zero outside $B_k$. Then it
is easily seen that $(-\De)^s_{K_k} w_k\ge w_k^p$ in $B_k$. Passing to the limit using again Lemma 5 of \cite{CS2},
we arrive at $(-\De)^s v \ge v^p$ in $\R^N$, which contradicts Theorem 1.3 in \cite{FQ2} since
$p<\frac{N}{N-2s}$. This concludes the proof.
\end{proof}

\bigskip

\begin{obss}\label{comentario} {\rm \

\noindent (a) With a minor modification in the above proof, it can be seen that
the constants given by Lemma \ref{cotas-pqs} can be taken independent of the domain $\Om$ 
(cf. the proof of Theorem 2.3 in \cite{PQS}).

\medskip
\noindent (b) We expect Lemma \ref{cotas-pqs} to hold in the full range given by \eqref{subcritico}.
Unfortunately, this method of proof seems purely local and needs to be properly adapted to
deal with nonlocal equations. Observe that there is no information available for the functions
$v_k$ defined in \eqref{v_k} in $\Om\setminus B_k$, which makes it difficult to pass to the limit appropriately in the
equation satisfied by $v_k$.
}\end{obss}

\bigskip

We now come to the proof of the a priori bounds for positive solutions of \eqref{problema-grad}.

\medskip

\begin{proof}[Proof of Theorem \ref{cotas-grad}]
Assume that the conclusion of the theorem is not true. Then there exists a sequence of positive
solutions $u_k\in E_\sigma$ of \eqref{problema-grad} such
that $\| u_k \|_1^{(-\sigma)}\to +\infty$, where $\sigma$ satisfies \eqref{cond-sigma}. Define
$$
M_k(x)= d(x)^{-\sigma} u_k(x) + d(x)^{1-\sigma} |\nabla u_k(x)|.
$$
Now choose points $x_k\in \Om$ such that $M_k(x_k) \ge \sup_\Om
M_k -\frac{1}{k}$ (this supremum may not be achieved). 
Observe that our assumption implies $M_k(x_k)\to +\infty$.

Let $\xi_k$ be a projection of $x_k$ on $\p \Om$ and introduce the functions:
$$
v_k(y) = \frac{u_k(\xi_k + \mu_k y)}{\mu_k^\sigma M_k(x_k)}, \quad y\in D^k,
$$
where $\mu_k=M_k(x_k)^{-\frac{p-1}{2s+\sigma(p-1)}}\to 0$ and $D^{k}$ is the set defined in \eqref{Dk}. It is not hard to see that
\begin{equation}\label{eq-rescalada}
\left\{
\begin{array}{ll}
(-\De)_{K_k}^s v_k = v_k^p + h_k  & \hbox{in } D^k,\\[0.35pc]
\ \ v_k=0 & \hbox{in }\R^N \setminus D^k,
\end{array}
\right.
\end{equation}
where  $K_k(y) = K (\mu_k y)$ and
$$
h_k (y)\hspace{-1mm}=\hspace{-1mm}M_k(x_k)^{-\frac{2sp}{2s+\sigma (p-1)}} h(\xi_k+\mu_k y ,M_k(x_k)^\frac{2s}{2s+\sigma(p-1)} v_k,
M_k(x_k)^\frac{2s+p-1}{2s+\sigma(p-1)} \nabla v_k).
$$
By assumption \eqref{crec-h-2} on $h$, it is readily seen that $h_k$ verifies the inequality
$|h_k|\le C M_k(x_k) ^{-\bar \gamma} (1+ v_k^r+ |\nabla v_k|^t)$ for
some positive constant $C$ independent of $k$, where
$$
\bar \gamma=\frac{2sp}{2s+\sigma(p-1)} -\frac{\max\{2sr, (2s+p-1)t\}}{2s+\sigma(p-1)}>0.
$$
Moreover, the functions $v_k$ verify
$$
\mu_k^\sigma d(\xi_k+\mu_k y)^{-\sigma} v_k(y)+ \mu_k^{\sigma-1} d(\xi_k+\mu_k y)^{1-\sigma} |\nabla v_k(y)|
=\frac{M_k(\xi_k+\mu_k y) }{M_k(x_k)}.
$$
Then, using that  $\mu_k ^{-1} d(\xi_k+\mu_k y)={\rm dist}(y,\p D^k)=:d_k(y)$ and the choice of the points
$x_k$, we obtain for large $k$
\begin{equation}\label{eq-normal-1}
d_k(y)^{-\sigma} v_k(y)+d_k(y)^{1-\sigma} |\nabla v_k(y)| \le 2 \quad \mbox{ in } D^k
\end{equation}
and
\begin{equation}\label{eq-normal-2}
d_k(y_k)^{-\sigma} v_k(y_k)+d_k(y_k)^{1-\sigma} |\nabla v_k(y_k)| =1,
\end{equation}
where, as in the proof of Theorem \ref{cotas}, $y_k :=\mu_k^{-1}(x_k-\xi_k)$.

Next, since $u_k$ solves \eqref{problema-grad}, we may use Lemma \ref{cotas-pqs} 
to obtain that $M_k(x_k) \le C d(x_k)^{-\sigma}(1+d(x_k)^{-\frac{2s}{p-1}})$
for some positive constant independent of $k$, which implies $d(x_k) \mu_k^{-1}\le C$.
This bound immediately entails that (passing to subsequences) $x_k\to x_0\in \p\Om$ and
$|y_k|=d(x_k)\mu_k^{-1}\to d\ge 0$ (in particular the points $\xi_k$ are uniquely determined
at least for large $k$). Assuming that the outward
unit normal to $\p \Om$ at $x_0$ is $-e_N$, we also obtain then that $D^k \to \R^N_+$ as
$k \to +\infty$.

We claim that $d>0$. To show this, notice that from \eqref{eq-rescalada} and
\eqref{eq-normal-1}  we have $(-\De)^s_{K_k}
v_k \le C d_k^{(\sigma-1)t}$ in $D^k$, for some constant $C$ not depending on $k$.
By our choice of $\sigma$ and $t$, we get that
\begin{equation}\label{cond-sigma2}
\sigma>\frac{t-2s}{t}.
\end{equation}
That is, we have
\begin{equation}\label{sigma3}
s<(1-\sigma)t<2s,
\end{equation}
so that Lemma \ref{lema-barrera-2} can
be applied to give $\de>0$ and a positive constant $C$ such that
\begin{equation}\label{these1}
v_k(y) \le C d_k(y)^{2s+(\sigma-1) t}, \quad \hbox{when } d_k(y) <\de.
\end{equation}
Moreover, since $1<t<2s$, \eqref{cond-sigma2} in particular implies that
\begin{equation}\label{sigma2}
\sigma>\frac{t-2s}{t-1},
\end{equation}
and, therefore, $-\sigma+2s+(\sigma-1)t =\sigma(t-1)+2s-t >0.$ Thus,  by \eqref{eq-normal-1} we have
$$\mbox{$v_k (y) \le 2 d_k(y)^{\sigma}\le 2 \de^{\sigma-2s-(\sigma-1)t} d_k(y)^{2s+(\sigma-1)t}$ when
$d_k(y)\ge \de$.}$$
Hence $\| v_k\|_0^{(-2s-(\sigma-1)t)}$ is bounded.
We can then use Lemma \ref{lema-regularidad}, with $\theta=(1-\sigma)t$, to obtain that
\begin{equation}\label{these2}
|\nabla v_k(y)| \le C d_k(y)^{2s+(\sigma-1) t-1} \quad \hbox{in } D^k,
\end{equation}
where $C$ is also independent of $k$. Taking inequalities \eqref{these1} and \eqref{these2} 
in \eqref{eq-normal-2}, we deduce
$$
1 \le C d_k(y_k) ^{-\sigma+2s+(\sigma-1)t},
$$
thus, by \eqref{sigma2} we see that $d_k (y_k)$ is bounded away from zero. Hence, by \eqref{cero}, 
$|y_k|$ also is, so that $d>0$, as claimed.

Finally, we can use \eqref{est-c1a} together with Ascoli-Arzel\'a's theorem and a
diagonal argument to obtain that $v_k \to v$ in $C^1_{\rm loc}(\R^N_+)$,
where by \eqref{eq-normal-2}, the function $v$ verifies $d^{-\sigma} v(y_0)+ d^{1-\sigma} |\nabla v(y_0)|=1$ for
some $y_0\in \R^N_+$, hence it is nontrivial and $v(y) \le C y_N ^{2s+(\sigma-1)t}$ if $0<y_N<\de$.
Thus $v\in C(\R^N)$ and $v=0$ outside $\R^N_+$. Passing to the limit in \eqref{eq-rescalada} with
the aid of Lemma 5 in \cite{CS2} and using that $K$ is continuous at zero with $K(0)=1$, we obtain
$$
\left\{
\begin{array}{ll}
(-\De)^s v = v^p & \hbox{in } \R^N_+,\\[0.25pc]
v=0 & \hbox{in } \R^N \setminus \R^N_+.
\end{array}
\right.
$$
Using again bootstrapping and the strong maximum principle we have
$v>0$ and $v\in C^\infty(\R^N_+)$, therefore it is a classical solution.  Moreover, by
Lemma \ref{cotas-pqs}, we also see that $v(y)\le C y_N^{-\frac{2s}{p-1}}$ in $\R^N_+$, so that
$v$ is bounded. This is a contradiction with Theorem 1.2 in \cite{FW}
(see also \cite{QX}), because
we are assuming $p<\frac{N}{N-2s} <\frac{N-1+2s}{N-1-2s}$. 
The proof is therefore concluded.
\end{proof}

\medskip

\section{Existence of solutions}
\setcounter{section}{4}
\setcounter{equation}{0}

This final section is devoted to the proof of our existence results, Theorems
\ref{th-1} and \ref{th-grad}. Both proofs are very similar, only that
that of Theorem \ref{th-grad} is slightly more involved. Therefore we only show this one.

Thus we assume $s>\frac{1}{2}$. Fix $\sigma$ verifying \eqref{cond-sigma} and 
consider the Banach space $E_\sigma$, defined in \eqref{E}, which is an ordered
Banach space with the
cone of nonnegative functions $P=\{u\in E_\sigma:\ u\ge 0 \hbox{ in }\Om\}$. For the
sake of brevity, we will drop the subindex $\sigma$ throughout the rest of the section
and will denote $E$ and $\| \cdot \|$ for the space and its norm.

We will assume that $h$ is nonnegative and verifies the growth condition in the statement of
Theorem \ref{th-grad}:
\begin{equation}\label{hipo-h-2}
h(x,z,\xi) \le C (|z|^r +|\xi|^t), \quad x\in \Om,\ z\in \R, \ \xi \in \R^N,
\end{equation}
where $1<r<p$ and $1<t<\frac{2sp}{2s+p-1}$. Observe that for every $v\in P$ we
have
\begin{equation}\label{h}
h(x,v(x),\nabla v(x)) \le C (\|v\|) d(x)^{(\sigma-1)t}.
\end{equation}
Moreover, by \eqref{sigma3} we may apply Lemma \ref{lema-existencia} to deduce that the problem
$$
\left\{
\begin{array}{ll}
(-\De)_K^s u = v^p + h(x,v,\nabla v) & \hbox{in }\Om,\\[0.35pc]
\ \ u=0 & \hbox{in }\R^N \setminus \Om,
\end{array}
\right.
$$
admits a unique nonnegative solution $u$, with $\|u \|_0^{(-\sigma)}<+\infty$. By Lemma \ref{lema-regularidad}
we also deduce $\| \nabla u \|_0^{(1-\sigma)}<+\infty$. Hence $u\in E$. In this way, we can
define an operator $T: P \to P$ by means of $u=T(v)$. It is clear that nonnegative solutions of \eqref{problema}
in $E$ coincide with the fixed points of this operator.

\medskip

We begin by showing a fundamental property of $T$.

\begin{lema}\label{lema-compacidad}
The operator $T: P  \to P$ is compact.
\end{lema}

\begin{proof}
We show continuity first: let $\{u_n\} \subset P$ be such that $u_n\to u$ in $E$. In particular,
$u_n\to u$ and $\nabla u_n\to \nabla u$ uniformly on compact sets of $\Om$, so that the continuity of
$h$ implies
\begin{equation}\label{conv-unif}
h(\cdot,u_n,\nabla u_n) \to h(\cdot,u,\nabla u) \hbox{ uniformly on compact sets of }\Om 
\end{equation}
Moreover, since $u_n$ is bounded in $E$, similarly as in \eqref{h} we also have that
$h(\cdot,u_n,\nabla u_n) \le C d^{(\sigma-1)t}$ in $\Om$, for a constant that does not depend on $n$
(and the same is true for $u$ after passing to the limit). This implies
\begin{equation}\label{claim!}
\sup_\Om d^\theta |h(\cdot,u_n,\nabla u_n)- h(\cdot,u,\nabla u)| \to 0,
\end{equation}
for every $\theta>(1-\sigma)t>s$. Indeed, if we take $\e>0$ then
$$
d^\theta |h(\cdot,u_n,\nabla u_n)- h(\cdot,u,\nabla u)|\leq C d^{\theta-(1-\sigma)t}
\le C \de^{\theta-(1-\sigma)t}\le \e,
$$
if $d\le \de$, by choosing a small $\de$. When $d\ge \de$,
$$
d^\theta |h(\cdot,u_n,\nabla u_n)- h(\cdot,u,\nabla u)|\leq (\sup_\Om d)^\theta
|h(\cdot,u_n,\nabla u_n)- h(\cdot,u,\nabla u)| \le \e,
$$
just by choosing $n\ge n_0$, by \eqref{conv-unif}. This shows \eqref{claim!}.

From Lemmas \ref{lema-existencia} and \ref{lema-regularidad} for every $(1-\sigma)t<\theta<2s$, we obtain
$$
\sup_\Om d^{\theta-2s} |T(u_n)-T(u)| + d^{\theta-2s+1} |\nabla (T(u_n)-T(u))| \to 0.
$$
The desired conclusion follows by choosing $\theta$ such that
$$
(1-\sigma) t < \theta \le 2s-\sigma.
$$
This shows continuity.

To prove compactness, let $\{u_n\} \subset P$ be bounded. As we did before,
$h(\cdot,u_n,\nabla u_n)
\le C d^{(\sigma-1)t}$ in $\Om$. By \eqref{est-c1a} we obtain that
for every $\Om' \subset \subset \Om$ the $C^{1,\beta}$ norm of $T(u_n)$ in $\Om'$ is
bounded. Therefore, we may assume by passing to a subsequence that $T(u_n)\to v$
in $C^1_{\rm loc} (\Om)$.

From Lemmas \ref{lema-existencia} and \ref{lema-regularidad}  we deduce that $T(u_n) \le Cd^{(\sigma-1)t+2s}$,
$|\nabla T(u_n)| \le Cd^{(\sigma-1)t+2s-1}$ in $\Om$, and the same estimates hold for
$v$ and $\nabla v$ by passing to the limit. Hence
$$
\sup_\Om d^{-\sigma} |T(u_n)-v| + d^{1-\sigma} |\nabla (T(u_n)-v)|  \to 0,
$$
which shows compactness. The proof is concluded.
\end{proof}

\bigskip

The proof of Theorem \ref{th-grad} relies in the use of topological degree, with the
aid of the bounds provided by Theorem \ref{cotas-grad}. The essential tool is the
following well-known result (see for instance Theorem 3.6.3 in \cite{Ch}).

\begin{teorema}\label{th-chang}
Suppose that $E$ is an ordered Banach space with positive cone $P$, and $U\subset P$ is an open
bounded set containing 0. Let $\rho>0$ be such that $B_\rho(0)\cap P\subset U$. Assume $T: U\to P$
is compact and satisfies

\begin{itemize}

\item[(a)] for every $\mu\in [0,1)$, we have $u\ne \mu T(u)$ for every $u \in P$ with $\| u \|=\rho$;

\smallskip
\item[(b)] there exists $\psi \in P\setminus \{0\}$ such that $u-T(u) \ne t \psi$, for every
$u\in \p U$, for every $t\ge 0$.

\end{itemize}
Then $T$ has a fixed point in $U \setminus B_\rho(0)$.
\end{teorema}

\bigskip

The final ingredient in our proof is some knowledge on the principal eigenvalue for the
operator $(-\De)^s_K$. The natural definition of such eigenvalue in our context
resembles that of \cite{BNV} for linear second order elliptic operators, that is:
\begin{equation}\label{eigenvalue}
\la_1 :=\sup\left\{  \la\in \R:
\begin{array}{cc}
\hbox{ there exists } u\in C(\R^N),\ u>0 \hbox{ in } \Om, \hbox{ with } \\[0.25pc]
u=0 \hbox{ in } \R^N\setminus \Om \hbox{ and } (-\De)^s_K u \ge \la u \hbox{ in } \Om
\end{array}
\right\}.
\end{equation}
At the best of our knowledge, there are no results available for the eigenvalues
of $(-\De)^s_K$, although it seems likely that the first one will enjoy the usual
properties (see \cite{QS}).

For our purposes here, we only need to show the finiteness of $\la_1$:

\begin{lema}\label{lema-auto}
$\la_1<+\infty$.
\end{lema}

\begin{proof}
We begin by constructing a suitable subsolution.
The construction relies in a sort of ``implicit" Hopf's principle (it is to be noted that Hopf's principle
is not well understood for general kernels $K$ verifying \eqref{elipticidad}; see for instance Lemma 7.3 in
\cite{RO} and the comments after it). However, a relaxed version is enough for our purposes.

Let $B'\subset \subset B\subset \subset \Om$ and consider the unique solution $\phi$ of
$$
\left\{
\begin{array}{ll}
(-\De)_K^s \phi = 0 & \hbox{in } B\setminus B',\\[0.35pc]
\ \ \phi = 1 & \hbox{in } B',\\[0.25pc]
\ \ \phi = 0 & \hbox{in }\R^N \setminus B.
\end{array}
\right.
$$
given for instance by Theorem 3.1 in \cite{FQ}, and the unique viscosity solution of
$$
\left\{
\begin{array}{ll}
(-\De)_K^s v= \phi & \hbox{in } B,\\[0.35pc]
\ \ v = 0 & \hbox{in }\R^N \setminus B.
\end{array}
\right.
$$
given by the same theorem. By Lemma \ref{PFM} we have both $\phi>0$ and $v>0$ in $B$, so that
there exists $C_0>0$ such that $C_0 v \ge \phi$ in $B'$.  Hence by comparison
$C_0 v \ge \phi$ in $\R^N$. In particular,
\begin{equation}\label{extra4_1}
\mbox{$(-\De)^s_K v \le C_0 v$ in $B$.}
\end{equation}
We claim that $\la_1\le C_0$. Indeed, if we assume $\la_1>C_0$, then
there exist $\la>C_0$ and a
positive function $u\in C(\R^N)$ vanishing outside $\Om$ such that
\begin{equation}\label{extra4_2}
\mbox{$(-\De)^s_K u \ge \la u$ in $\Om$.}
\end{equation}
Since $u>0$ in $\overline{B}$, the number
$$
\omega =\sup _B \frac{v}{u}
$$
is finite. Moreover, $\omega u \ge v$ in $\R^N$. Observe that, since we are assuming $\la>C_0$, by
\eqref{extra4_1} and \eqref{extra4_2} it follows that
$$
\left\{
\begin{array}{ll}
(-\De)^s_K (\omega u-v)\ge 0 & \hbox{in } B,\\[0.35pc]
\ \ \omega u-v > 0 & \hbox{in }\R^N \setminus B.
\end{array}
\right.
$$
Hence the strong maximum principle (Lemma \ref{PFM}) implies
$\omega u -v>0$ in $\overline{B}$. However this would imply $(\omega-\e) u >v$ in
$\overline{B}$ for small $\e$, contradicting the definition of $\omega$. Then $\la_1\le C_0$
and the lemma follows.
\end{proof}

\bigskip

Now we are in a position to prove Theorem \ref{th-grad}.

\bigskip

\begin{proof}[Proof of Theorem \ref{th-grad}]
As already remarked, we will show that Theorem \ref{th-chang} is applicable to the operator
$T$ in $P\subset E$.

Let us check first hypothesis (a) in Theorem \ref{th-chang}. Assume we have $u=\mu T(u)$
for some $\mu \in [0,1)$ and $u\in P$. This is equivalent to
$$
\left\{
\begin{array}{ll}
(-\De)_K^s u = \mu (u^p + h(x,u,\nabla u)) & \hbox{in }\Om,\\[0.35pc]
\ \ u=0 & \hbox{in }\R^N \setminus \Om.
\end{array}
\right.
$$
By our hypotheses on $h$ we get that the right hand side of the previous
equation can be bounded by
$$
\begin{array}{rl}
\mu (u^p+h(x,u,\nabla u)) \hspace{-2mm} & \le d^{\sigma p} \| u\| ^p + C_0( d^{\sigma r} \| u\|^r
+ d^{(\sigma-1)t} \| u\| ^t)\\[0.25pc]
& \le C d^{(\sigma-1) t} ( \| u\| ^p + \| u\|^r + \| u\| ^t).
\end{array}
$$
Therefore, by Lemmas \ref{lema-existencia} and \ref{lema-regularidad} and \eqref{sigma3}, we have $\| u \| \le C
( \| u\| ^p + \| u\|^r + \| u\| ^t)$. Since $p,r,t>1$, this implies that $\| u \| > \rho$ for some small
positive $\rho$. Thus there are no solutions of $u=\mu T(u)$ if $\| u \|=\rho$
and $\mu\in [0,1)$, and (a) follows.

To check (b), we take $\psi \in P$ to be the unique solution of the problem:
$$
\left\{
\begin{array}{ll}
(-\De)_K^s \psi = 1 & \hbox{in }\Om,\\[0.35pc]
\ \ \psi = 0 & \hbox{in }\R^N \setminus \Om
\end{array}
\right.
$$
given by Theorem 3.1 in \cite{FQ}.
We claim that there are no solutions in $P$ of the equation $u-T(u)=t \psi$ if $t$ is large enough.
For that purpose we note that this equation is equivalent to
\begin{equation}\label{problema-t}
\left\{
\begin{array}{ll}
(-\De)_K^s u = u^p + h(x,u,\nabla u)+t & \hbox{in }\Om,\\[0.35pc]
\ \ u=0 & \hbox{in }\R^N \setminus \Om.
\end{array}
\right.
\end{equation}
Fix $\mu> \la_1$, where $\la_1$ is given by \eqref{eigenvalue}. Using the nonnegativity of $h$,
and since $p>1$,  there exists a positive constant $C$ such that $u^p + h(x,u,\nabla u)+t \ge \mu u - C +t$.
If $t\ge C$, then $(-\De)^s_K u \ge \mu u$ in $\Om$, which is against the choice of $\mu$ and
the definition of $\la_1$. Therefore $t< C$, and \eqref{problema-t} does not admit
positive solutions in $E$ if $t$ is large enough.

Finally, since $h+t$ also verifies condition \eqref{crec-h-2} for $t\le C$,
we can apply Theorem \ref{cotas-grad} to obtain that the solutions of
\eqref{problema-t} are a priori bounded, that is, there exists $M > \rho$ such that
$\| u\| < M$ for every positive solution of \eqref{problema-t} with $t\ge 0$.
Thus Theorem \ref{th-chang} is applicable with $U=B_M(0)\cap P$ and the existence
of a solution in $P$ follows. This solution is positive by Lemma \ref{PFM}. The proof
is concluded.
\end{proof}

\bigskip

\noindent {\bf Acknowledgements.} B. B. was partially supported by a postdoctoral fellowship
given by Fundaci\'on Ram\'on Areces (Spain) and MTM2013-40846-P, MINECO.
L. D. P. was partially supported by PICT2012 0153 from ANPCyT (Argentina).
J. G-M and A. Q. were partially supported by Ministerio de Ciencia e
Innovaci\'on under grant MTM2011-27998 (Spain) and Conicyt MEC number 80130002.
A. Q. was also partially supported by Fondecyt Grant
No. 1151180 Programa Basal, CMM. U. de Chile and Millennium Nucleus
Center for Analysis of PDE NC130017.

\end{document}